\documentclass[oneside, a4paper,11pt,reqno]{amsart}
\usepackage{amsfonts,amsmath}
\usepackage[T1]{fontenc}

\newtheorem{hypo}{Hypothesis}

\newtheorem{prop}[hypo]{Proposition}

\newtheorem{thm}[hypo]{Theorem}

\newtheorem{lem}[hypo]{Lemma}

\def\C{\mathcal{C}}
\def\D{\mathcal{D}}

\def\I{\mathcal{I}}
\def\O{\mathcal{O}}
\def\E{\mathcal{E}}

\def\PP{\mathbb{P}}
\def\RR{\mathbb{R}}
\def\ZZ{\mathbb{Z}}
\def\CC{\mathbb{C}}
\def\EE{\mathbb{E}}
\def\NN{\mathbb{N}}

\newcommand {\pare}[1] {\left( {#1} \right)}
\newcommand {\cro}[1] {\left[ {#1} \right]}
\def \ind {\hbox{ 1\hskip -3pt I}}
\newcommand {\refeq}[1] {(\ref{#1})}
\newcommand {\va}[1] {\left| {#1} \right|}
\newcommand {\acc}[1] {\left\{ {#1} \right\}}
\newcommand {\floor}[1] {\left\lfloor {#1} \right\rfloor}

\newcommand {\ceil}[1] {\left\lceil {#1} \right\rceil}

\title[Limit theorems for 1-d and 2-d RWRS]
       {Limit theorems
   for one and two-dimensional random walks in random scenery}
\author{Fabienne Castell} 
\address{LATP, UMR CNRS 6632. Centre de Math\'ematiques et Informatique.
Universit\'e Aix-Marseille I. 39, rue Joliot Curie. 13 453 Marseille Cedex
13. France.}
\email{Fabienne.Castell@cmi.univ-mrs.fr}

\author{Nadine Guillotin-Plantard} 
\address{Institut Camille Jordan, CNRS UMR 5208, Universit\'e de Lyon, Universit\'e Lyon 1, 43, Boulevard du 11 novembre 1918, 69622 Villeurbanne, France.}
\email{nadine.guillotin@univ-lyon1.fr}

\author{Fran\c{c}oise P\`ene}
\address{Universit\'e Europ\'eenne de Bretagne, Universit\'e de Brest,
D\'epartement de Math\'ematiques, 29238 Brest cedex, France}
\email{francoise.pene@univ-brest.fr}

\keywords{Random walk in random scenery; local limit theorem; local time; stable process\\
This research was supported by the french ANR project MEMEMO2 and RANDYMECA}

\begin{document}

\begin{abstract} 
Random walks in random scenery are processes defined by  
$Z_n:=\sum_{k=1}^n\xi_{X_1+...+X_k}$, where $(X_k,k\ge 1)$ and $(\xi_y,y\in{\mathbb Z}^d)$
are two independent sequences of i.i.d. random variables with values
in ${\mathbb Z}^d$ and $\mathbb R$ respectively.
We suppose that the distributions of $X_1$ and $\xi_0$ belong to the normal basin of
attraction of stable distribution of index $\alpha\in(0,2]$ and $\beta\in(0,2]$. 
When $d=1$ and $\alpha\ne 1$, a functional limit theorem has been established in \cite{KestenSpitzer}
and a local limit theorem in \cite{BFFN}.
In this paper, we establish the convergence of
the finite-dimensional distributions and a local limit theorem
when $\alpha=d$ (i.e.
$\alpha = d=1$ or $\alpha=d=2$) and $\beta \in (0,2]$. 
Let us mention that functional limit theorems 
have been established in
\cite{bolthausen} and recently in \cite{DU} in 
the particular case where $\beta=2$ (respectively 
for $\alpha=d=2$ and $\alpha=d=1$). 
\end{abstract}

\maketitle

\section{Introduction}
  Random walks in random scenery (RWRS) 
are simple models of processes in disordered
media with long-range correlations. They have been used in a wide
variety of models in physics to study anomalous dispersion in layered
random flows \cite{matheron_demarsily},  diffusion with random sources, 
or spin depolarization in random fields (we refer the reader
to Le Doussal's review paper \cite{ledoussal} for a discussion of these
models). 

On the mathematical side, motivated by the construction of 
new self-similar processes with stationary increments, 
Kesten and Spitzer \cite{KestenSpitzer} and Borodin  \cite{Borodin, Borodin1} 
introduced RWRS in dimension one and proved functional limit theorems. 
This study has been completed in many works, in particular in
\cite{bolthausen} and \cite{DU}.
These processes are defined as follows. Let $\xi:=(\xi_y,y\in \ZZ^d)$ and $X:=(X_k,k\ge 1)$ 
be two independent sequences of independent
identically distributed random variables taking values in $\RR$ and $\ZZ^d$ 
respectively. 
The sequence $\xi$ is called the {\it random scenery}. 
The sequence $X$ is the sequence of increments of the {\it random walk}  
$(S_n, n \geq 0)$
defined by $S_0:=0$ and  $S_n:=\sum_{i=1}^{n}X_i$, for $n\ge 1$. 
The {\it random walk in random scenery} $Z$ is 
then defined by
$$Z_0:=0\ \mbox{and}\ \forall n\ge 1,\ Z_n:=\sum_{k=0}^{n-1}\xi_{S_k}.$$ 
Denoting by  
$N_n(y)$ the local time of the random walk $S$~:
$$N_n(y):=\#\{k=0,...,n-1\ :\ S_k=y\} \, ,
$$
it is straightforward to see that 
$Z_n$ can be rewritten as $Z_n=\sum_y\xi_yN_n(y)$.

As in \cite{KestenSpitzer}, 
the distribution of $\xi_0$ is assumed to belong to the normal 
domain of attraction of a strictly stable distribution 
$\mathcal{S}_{\beta}$ of 
index $\beta\in (0,2]$, with characteristic function $\phi$ given by
$$\phi(u)=e^{-|u|^\beta(A_1+iA_2 \text{sgn}(u))}\quad u\in\mathbb{R},$$
where $0<A_1<\infty$ and $|A_1^{-1}A_2|\le |\tan (\pi\beta/2)|$. 
We will denote by $\varphi_\xi$ the characteristic function of the $\xi_x$'s. 
When $\beta > 1$, this implies that $\EE[\xi_0] = 0$. When $\beta = 1$, 
we will further assume the symmetry condition
\begin{equation} 
\label{symetrie}
\sup_{ t > 0} \va{\EE\cro{\xi_0 \ind_{\{\va{\xi_0} \le t\}}}} < +\infty \, .
\end{equation}
Under these conditions (for $\beta\in(0;2]$), 
there exists $C_\xi>0$ such that we have
\begin{equation}
\label{queue}
\forall t>0,\ \ \PP \pare{ \va{\xi_0} \ge t} \le C_\xi t^{-\beta}.
\end{equation}

\noindent Concerning the random walk, the distribution of $X_1$ is
assumed to belong to the normal basin of attraction of a stable
distribution ${\mathcal S}'_{\alpha}$ with index $\alpha\in (0,2]$.

\noindent Then the following weak convergences hold in the space  of 
c\`adl\`ag real-valued functions 
defined on $[0,\infty)$ and on $\mathbb R$ respectively,  endowed with the 
Skorohod $J_1$-topology (see \cite[chapter 3]{billingsley})~:
$$\left(n^{-1/\alpha} S_{\lfloor nt\rfloor}\right)_{t\geq 0}   
\mathop{\Longrightarrow}_{n\rightarrow\infty}
^{\mathcal{L}} \left(U(t)\right)_{t\geq 0}$$
$$\mbox{\rm and} \  \  \   \left(n^{-\frac{1}{\beta}} 
\sum_{k=0}^{\lfloor nx\rfloor}\xi_{k e_1}\right)_{x\in\mathbb R}
   \mathop{\Longrightarrow}_{n\rightarrow\infty}^{\mathcal{L}} 
\left(Y(x)\right)_{x\in\mathbb R},
\mbox{ with } e_1=(1,0,\cdots,0) \in \ZZ^d \, ,$$
where $U$ and $Y$ are two independent L\'evy processes such 
that $U(0)=0$, $Y(0)=0$, 
$U(1)$ has distribution $\mathcal{S}'_{\alpha}$, $Y(1)$ and 
$Y(-1)$ have distribution  $\mathcal{S}_\beta$.

\noindent
{\bf Functional limit theorem.}\\
  Our first result is concerned with a   
functional limit theorem for $(Z_{[nt]})_{t \ge 0}$. Intuitively speaking, 
\begin{itemize} 
\item when $\alpha < d$, the random
walk $S_n$ is transient, its range is of order $n$, 
and $Z_n$ has the same behaviour as a sum of
about $n$ independent random variables 
with the same distribution as the variables 
$\xi_x$. Therefore, $n^{-1/\beta} (Z_{[nt]})_{t \ge 0}$ weakly converges 
in  the space  $D([0,\infty))$ of 
c\`adl\`ag functions endowed with the 
Skorohod $J_1$-topology, to a multiple of the process $(Y_t)$, as proved in
\cite{Borodin1}; 
\item  when $\alpha > d$ (i.e $d=1$ and $1< \alpha \le 2$), the random
walk $S_n$ is recurrent, its range is of order $n^{1/\alpha}$, its local
times are of order $n^{1-1/\alpha}$, so that $Z_n$ is of order 
$n^{1-\frac{1}{\alpha}+\frac{1}{\alpha \beta}}$. In this situation, 
\cite{Borodin} and \cite{KestenSpitzer} proved a functional limit
 theorem for 
$n^{-(1-\frac{1}{\alpha}+\frac{1}{\alpha \beta})} (Z_{[nt]})_{t \ge 0}$ in the space 
$\CC([0,\infty))$ of continuous  functions endowed with the 
uniform topology , 
the limiting process being a self-similar process, but not a stable one.
\item when $\alpha = d$ (i.e. $\alpha=d=1$, or $\alpha = d = 2$),  
$S_n$ is recurrent,  its range is of order $n/\log(n)$, its 
local times are of order $\log(n)$ so that  $Z_n$ is of order 
$n^{\frac{1}{\beta}} \log(n)^{\frac{\beta-1}{\beta}}$. In this situation, a 
functional limit theorem in the space of continuous functions was proved in \cite{bolthausen} for 
$d=\alpha=\beta=2$, and in \cite{DU} for $d=\alpha=1$ and $\beta=2$.
\end{itemize}
Our first result gives a limit theorem for $\alpha=d$
(and so $ d\in \acc{1,2}$)
and for any value of $\beta \in (0;2)$ in the finite distributional sense.

\begin{thm}\label{thmFLT}
Let us assume that $\beta\in(0;2]$ and that
\begin{itemize}
\item[(a)] either $d=2$ and $X_1$ is centered, square integrable with invertible 
variance matrix $\Sigma$ and then we define $A:=2\sqrt{\det\Sigma}$;
\item[(b)] or $d=1$ and $\left(\frac{S_n}n\right)_n$ converges
in distribution to a random variable with characteristic function given by 
$t\mapsto \exp(-a |t|)$ with $a>0$ and then we define $A:=a$.
\end{itemize}
Then, the finite-dimensional distributions of the sequence of random variables
$$\left(\left(\frac{Z_{[nt]}}{n^{1/\beta}  \log(n)^{(\beta-1)/\beta}}  
\right)_{t\geq 0}\right)_{n\ge 2} $$ 
converges to the finite-dimensional distributions of the process 
$$\left(\tilde Y_t:=\left(\frac{\Gamma(\beta +1)}{(\pi A)^{\beta-1}}\right)^{1/\beta}  Y(t)
    \right)_{t\geq 0}.$$
Moreover, if $\beta<2$, the sequence   
$$\left(\left(\frac{Z_{[nt]}}{n^{1/\beta}  \log(n)^{(\beta-1)/\beta}}  
\right)_{t\geq 0}\right)_{n\ge 2} $$ is not
tight in  $\D([0,\infty))$ endowed with the $J_1$-topology.
\end{thm}
\noindent
\vspace{.5cm}
\noindent
{\bf Local limit theorem.}\\
Our next results concern a local limit theorem for $(Z_n)_n$. 
The $d=1$ case was treated in 
\cite{BFFN} for $\alpha \in (0;2]\backslash \acc{1}$ and all values 
of $\beta \in (0;2]$. Here, we complete this study by proving 
a local limit theorem for $\alpha=d=1$ (and 
$\beta\in(0;2]$). By a direct adaptation of the proof of this result,
we also establish a local limit theorem for $\alpha=d=2$
(we just adapt the definition of "peaks", see section 3.5). 
Let us notice that the same adaptation can 
be done from \cite{BFFN} (case $\alpha<1$) to get local limit theorems 
for $d\geq 2$, $\alpha<d$ and $\beta\in(0;2]$.

We give two results corresponding respectively to
the case when $\xi_0$ is lattice and to the 
case when it is strongly non-lattice.
We denote by $\varphi_\xi$ the characteristic function of
$\xi_0$.
\begin{thm}\label{thmTLL}
Assume that $\xi_0$ takes its values in $\mathbb Z$ .
Let $d_0\ge 1$ be the integer such that $\{u\, :\, |\varphi_\xi(u)|=1\}=\frac{2\pi}{d_0}\mathbb Z$.
Let $b_n := n^{1/\beta}(\log(n))^{(\beta-1)/\beta}$. Under the previous assumptions on the random walk and on the scenery, 
for $\alpha = d \in \acc{1,2}$, for every $\beta \in (0,2]$, and 
for every $x\in \RR$,

$\bullet$ if ${\mathbb P}\left(n\xi_0-\floor{b_n x}\notin d_0{\mathbb Z}\right)=1$, 
then ${\mathbb P} \pare{Z_n= \floor{b_n x}}=0$;

$\bullet$ if ${\mathbb P}\pare{n\xi_0-\floor{b_n x}\in d_0{\mathbb Z}}=1$, then 
$$\PP\left(Z_n= \floor{  b_n  x}\right ) = 
    d_0\frac{C(x)}{  n^{1/\beta}(\log(n))^{(\beta-1)/\beta} }+ o( n^{-1/\beta}(\log(n))^{-(\beta-1)/\beta} )$$ 
uniformly in $x\in{\mathbb R}$, where $C(\cdot)$ is the density function of $\tilde Y_1$.
\end{thm}
\begin{thm}\label{thmTLL2}
Assume now that $\xi_0$ is strongly non-lattice which means that
$$\limsup_{|u|\rightarrow +\infty}|\varphi_\xi(u)|<1.$$
We still assume that $\alpha=d\in\{1,2\}$ and $\beta\in(0;2]$.
Then, for every $x,a,b\in\mathbb R$ such that $a<b$, we have
$$\lim_{n\rightarrow +\infty} b_n{\mathbb P}\pare{Z_n\in[b_nx+a;b_nx+b]}
   =C(x)(b-a),
 $$
with $b_n := n^{1/\beta}(\log(n))^{(\beta-1)/\beta}$
and where $C(\cdot)$ is the density function of $\tilde Y_1$.
\end{thm}

\section{Proof of the limit theorem}
\noindent Before proving the theorem, we prove some technical lemmas. 
For any real number $\gamma >0$, any integer $m\ge 1$, any
$\theta_1,\ldots,\theta_m\in {\mathbb R}$, 
any $t_0=0<t_1<\ldots<t_m$, we consider the sequences of random variables 
$(L_n(\gamma))_{n\ge 2}$ and $(L'_n(\gamma))_{n\ge 2}$ defined by
$$L_n(\gamma):=\frac{1} {n (\log n)^{\gamma-1}}\sum_{x\in {\mathbb Z}^d} 
   \left\vert \sum_{i=1}^m \theta_i   ( N_{[nt_i]}(x) - N_{[nt_{i-1}]}(x) ) 
  \right\vert^{\gamma} $$
and
$$L'_n(\gamma):=\frac{1} {n (\log n)^{\gamma-1}}\sum_{x\in {\mathbb Z}^d} 
   \left\vert \sum_{i=1}^m \theta_i   ( N_{[nt_i]}(x) - N_{[nt_{i-1}]}(x) ) 
  \right\vert^{\gamma} \text{sgn}\left( \sum_{i=1}^m \theta_i   
    ( N_{[nt_i]}(x) - N_{[nt_{i-1}]}(x) ) \right).$$
\begin{lem}\label{tech1}
For any real number $\gamma > 0$, any integer $m\ge 1$, any
$\theta_1,\ldots,\theta_m\in {\mathbb R}$, 
any $t_0=0<t_1<\ldots<t_m$, the following convergences hold ${\mathbb P}$-almost surely
\begin{equation}\label{Ln}
\lim_{n\rightarrow +\infty}L_{n} (\gamma)
    =\frac{\Gamma(\gamma+1)}{(\pi A)^{\gamma-1}} \sum_{i=1}^m |\theta_i|^{\gamma} (t_i -t_{i-1})
\end{equation}
and
\begin{equation}\label{L'n}
\lim_{n\rightarrow +\infty}L'_{n} (\gamma)
    =\frac{\Gamma(\gamma+1)}{(\pi A)^{\gamma-1}} \sum_{i=1}^m|\theta_i|^{\gamma}
   \text{sgn}(\theta_i) (t_i -t_{i-1}).
\end{equation}
\end{lem}
\begin{proof}
We fix an integer $m\ge 1$ and $2m$ real numbers $\theta_1,\ldots,\theta_m,t_1,...,t_m$
such that $0<t_1<\ldots<t_m$ and we set $t_0:=0$.
To simplify notations, we write $b_{i,n}(x):= N_{[n t_i ]} (x)
-N_{[n t_{i-1}]}(x) $.
Following  the techniques developed in \cite{Cerny}, we first have to prove 
\refeq{Ln} and \refeq{L'n}  for integer $\gamma$:  
for every integer $k\ge 1$,
$\PP$-almost surely, as $n$ goes to infinity, we have
\begin{equation}\label{pair}
\frac{1}{n (\log n)^{k-1}} \sum_{x\in{\mathbb Z}^d}
 \left( \sum_{i=1}^m \theta_i b_{i,n}(x) \right)^{k}
\longrightarrow \frac{\Gamma(k+1)}{(\pi A)^{k-1}} 
\sum_{i=1}^m \theta_i ^{k} (t_i- t_{i-1}).
\end{equation}
Let us assume \refeq{pair} for a while, and let us end the proof of 
\refeq{Ln} and \refeq{L'n} for any positive real $\gamma$. 
\noindent Given the random walk $S:=(S_n)_{n}$, let $(U_n)_{n \ge 1}$ be a sequence of random variables with values in $\mathbb{Z}^d$, such that for all $n$, $U_n$ is a point chosen uniformly in 
the range of the random walk up to time $[nt_m]$, that is 
$$\PP(U_n=x \big| S) =  R_{[nt_m]}^{-1} {\bf 1}_{\{N_{[nt_m]}(x)\geq 1\}} , $$
with $R_k:=\#\{y\,:\, N_k(y)>0\}$.
Moreover, let $U'$ be a random variable with values in $\{1,\ldots,m\}$ and distribution 
$$\PP(U'=i)= (t_i- t_{i-1})/t_m$$ 
and let $T$ be a random variable with exponential distribution with parameter one and independent of $U'$.
\\
Then, for $\PP-$ almost every realization of the random walk $S$, the sequence
of random variables 
$$\left(W_n:=\frac{ \pi A}{\log(n)}  \sum_{i=1}^m \theta_i  b_{i,n}(U_n)\right)_n$$
converges in distribution to the random variable $W:=\theta_{U'} T$.
Indeed, the moment of order $k$ of $W_n$ given $S$ is
\begin{eqnarray*}
\EE(W_n^{k}\big| S) &=&\frac{(\pi A)^{k}}{ n(\log n)^{k-1}} \sum_{x\in\mathbb{Z}^d} \left( 
 \sum_{i=1}^m \theta_i b_{i,n}(x)\right)^{k} \frac{n}{ \log(n) R([nt_m])}.   
\end{eqnarray*}
Using \refeq{pair} and the fact that  $((\log n) R_n /n)_n $ converges almost surely to $\pi A$ (see \cite{DE51,LGR}), 
the moments $\EE(W_n^{k}\big| S)$ converges a.s. to 
$\EE(W^{k})=\Gamma(k+1)   \sum_{i=1}^m \theta_i ^{k} (t_i- t_{i-1})/ t_m$, which
proves the convergence in distribution of $(W_n)_n$ (given $S$) to $W$.
This ensure, in particular, the convergence
in distribution of $(|W_n|^\gamma)_n$ and of 
$(|W_n|^\gamma\text{sgn}(W_n))_n$ (given $S$)
to $|W|^\gamma$ and $|W|^\gamma\text{sgn}(W)$ respectively (for every real number
$\gamma\ge 0$  and for $\PP-$ almost every realization of the random walk $S$).  
Since any moment of $|W_n|$ can be bounded 
from above by an integer moment, we deduce that, for any $\gamma\geq 0$, 
we have $\mathbb P$-almost surely
$$ \lim_{n\rightarrow +\infty} \EE(|W_n|^{\gamma}\big| S)
     = \EE(|W|^{\gamma})\ \ \mbox{and}\ \
   \lim_{n\rightarrow +\infty} \EE(|W_n|^{\gamma}\text{sgn}(W_n)\big| S)
     = \EE(|W|^{\gamma}\text{sgn}(W)),$$ 
which proves lemma \ref{tech1}.\\*
Let us prove (\ref{pair}). Let $k\geq 1$.
According to Theorem 1 in \cite{Cerny} (proved for $\alpha=d=2$, but also
valid for $\alpha=d=1$), we have
\begin{equation}\label{EQ}
\forall i\in\{1,...,m\},\ \ \lim_{n\rightarrow +\infty}
    \frac 1{n(\log n)^{k-1}}\sum_{x\in{\mathbb Z}^d}(b_{i,n}(x))^{k}
     = \frac{\Gamma(k+1)}{(\pi A)^{k-1}}(t_i-t_{i-1}),\ {\mathbb P}-a.s..
\end{equation}
We define
\begin{equation}\label{pair2}
\Sigma_n(\theta_1,...,\theta_m)
   := \sum_{x\in {\mathbb Z}^d} 
   \left( \sum_{i=1}^m \theta_i  b_{i,n}(x)\right)^{k} - 
  \sum_{x\in{\mathbb Z}^d} \sum_{i=1}^m (\theta_i )^{k}  \left(b_{i,n}(x) \right)^{k} .
\end{equation}
According to (\ref{EQ}), it is enough to prove that $\PP-$a.s., 
$\Sigma_n(\theta_1,...,\theta_m)=o(n(\log n)^{k-1})$.
We observe that $\Sigma_n(\theta_1,...,\theta_m)$ is the sum
of the following terms
\begin{equation}\label{sumprod}
\sum_{x\in {\mathbb Z}^d} \prod_{j=1}^{k} \left(\theta_{i_j} b_{i_j,n}(x)\right).
\end{equation}
over all the $k$-tuple $(i_1,\ldots,i_{k})\in\{1,\ldots,m\}^{k}$, 
with at least two distinct indices.
We observe that
$$ |\Sigma_n(\theta_1,...,\theta_m)| \le \max(|\theta_1|,...,|\theta_m|)^{k}
     \Sigma_n(1,...,1).$$
But, we have
\begin{eqnarray*}
\Sigma_n(1,...,1)
   &=&\sum_{x\in {\mathbb Z}^d} 
   \left(  N_{[nt_m]}(x) 
  \right)^{k} - 
  \sum_{x\in{\mathbb Z}^d} \sum_{i=1}^m   \left( b_{i,n}(x)\right)^{k}\\ 
   &=& \sum_{x\in {\mathbb Z}^d} 
   \left(  N_{[nt_m]}(x) 
  \right)^{k} - 
   \sum_{i=1}^m   \sum_{x\in{\mathbb Z}^d}\left( b_{i,n}(x)\right)^{k} =
      o(n\log(n)^{k-1}),
\end{eqnarray*}
according to (\ref{EQ}).
\end{proof}
\begin{lem}\label{sup}
For any $\rho > 0$, 
$$\sup_{x\in {\mathbb Z}^d} N_n(x) = o(n^{\rho})\  \  \   \mbox{\rm  a.s.}.$$ 
\end{lem}
\begin{proof} See Lemma 2.5 in \cite{bolthausen}.
\end{proof}
\begin{proof}[Proof of Theorem \ref{thmFLT}]

Let an integer $m\ge 1$ and $2m$ real numbers $\theta_1,...,\theta_m,t_1,...,t_m$
such that $0<t_1<...<t_m$. We set $t_0:=0$.
Again, we use the notation $b_{i,n}(x):=N_{[nt_i]}(x)-N_{[nt_{i-1}]}(x)$.
Let us write $\tilde Z_n:= \frac 1{n^{1/\beta}(\log(n))^{(\beta-1)/\beta}}
     \sum_{i=1}^m\theta_i(Z_{[nt_i]}-Z_{[nt_{i-1}]})$.
We have to prove that
\begin{equation}\label{distribfinies}
{\mathbb E}[e^{i\tilde Z_n}] \rightarrow\prod_{i=1}^m\phi\left(
     \theta_i(t_i-t_{i-1})^{1/\beta}
     \left(\frac{\Gamma(\beta+1)}{(\pi A)^{\beta-1}}\right)^{1/\beta}  \right),
\end{equation}
as $n$ goes to infinity.
We observe that $\tilde Z_n= \frac 1{n^{1/\beta}(\log(n))^{(\beta-1)/\beta}}
     \sum_{x\in{\mathbb Z}^d}\sum_{i=1}^m\theta_i b_{i,n}(x)\xi_x$.
Hence we have
$${\mathbb E}[e^{i\tilde Z_n}\vert S]=\prod_{x\in{\mathbb Z}^d}\varphi_\xi\left(
     \frac{ \sum_{i=1}^m\theta_i b_{i,n}(x)}{n^{1/\beta}(\log(n))^{(\beta-1)/\beta}}
  \right). $$
Observe next that 
$$\left\vert \varphi_\xi(t)-\exp\left(-|t|^\beta(A_1+iA_2\textrm{sgn}(t)\right)\right\vert\le
|t|^\beta h(\vert t\vert) \quad \textrm{for all } t\in \RR,$$ 
with $h$ a continuous and monotone function on $[0,+\infty)$ vanishing in $0$. 
This implies in particular the existence of 
$\varepsilon_0>0$ and $\sigma>0$ such that 
$\max( |\varphi_{\xi}(t)|,  \exp\left(-A_1 |t|^\beta\right))\leq e^{-\sigma |t|^{\beta}}$ for any 
$t\in [-\varepsilon_0,\varepsilon_0]$. 
According to lemma \ref{sup}, ${\mathbb P}$-almost surely, for every $n$ large enough, we have
$$b_n:=\sup_x\frac{\vert\sum_{i=1}^m\theta_i b_{i,n}(x)\vert}
   {n^{1/\beta}(\log(n))^{(\beta-1)/\beta}}
   \le  \varepsilon_0 $$
and so
$$\left\vert{\mathbb E}[e^{i\tilde Z_n}\vert S]-\prod_{x\in{\mathbb Z}^d}
     e^{-\frac{\left|\sum_{i=1}^m\theta_i b_{i,n}(x)\right|^\beta}{n(\log(n))^{\beta-1}}
             (A_1+iA_2\text{sgn}\left( \sum_{i=1}^m\theta_ib_{i,n}(x)\right))}
 \right\vert $$
is less than $\sum_{x\in{\mathbb Z}^d}\frac{
\left|\sum_{i=1}^m\theta_i b_{i,n}(x)\right|^\beta}{n(\log(n))^{\beta-1}}
h(b_n)e^{-\sigma\left(\frac{\sum_{y\in\mathbb Z} \left\vert
       \sum_{i=1}^m\theta_ib_{i,n}(y)  \right\vert^\beta}{n(\log n)^{\beta-1}}-b_n^\beta\right)}$.
Hence, according to lemmas \ref{tech1} and \ref{sup}, ${\mathbb P}$-almost surely, we have
$$\lim_{n\rightarrow +\infty}{\mathbb E}[e^{i\tilde Z_n}\vert S] =
     e^{-\frac{\Gamma(\beta+1)}{(\pi A)^{\beta-1}}\sum_{i=1}^m
            |\theta_i|^\beta(t_i-t_{i-1}) (A_1+iA_2\text{sgn}\left(\theta_i\right))}
$$
which gives (\ref{distribfinies}) thanks to the Lebesgue dominated convergence theorem.
\\*
\\*
Finally we prove that the sequence 
$$\left(\left(\frac{Z_{[nt]}}{n^{1/\beta}  \log(n)^{(\beta-1)/\beta}} 
 \right)_{t\in [0;1]} \right)_{n\ge 2}$$ is not
tight in  $\D([0,\infty))$. It is enough to prove that it is not tight 
in $\D([0,1])$. To this aim, let 
$b_n = n^{1/\beta}  \log(n)^{(\beta-1)/\beta}$, and 
$(Z_n(t), t \in [0,1])$ denote the linear interpolation of $(Z_{[nt]}, 
t \in [0,1])$, i.e.
\[ Z_n(t) = Z_{[nt]} + (nt-[nt]) \xi_{S_{[nt]}} \, .\]
Then, $\forall \epsilon > 0$, 
\begin{eqnarray*}
 \PP \cro{\sup_{t \in [0,1]} \va{Z_n(t) - Z_{[nt]}} \ge \epsilon b_n}
 & =  &  \PP \cro{\max_{i=0}^{n -1} \va{\xi_{S_i}} \ge \epsilon b_n }
 \\
 & = & \PP \cro{\exists x \in \acc{S_0, \cdots, S_{n-1}}  \mbox{ s.t }  \va{\xi_x} \ge \epsilon b_n}
  \\
 & \le & \EE(\#\acc{S_0, \cdots, S_{n-1}})  \PP \cro{ \va{\xi_0} \ge \epsilon b_n}
 \\
 & \le & C \frac{n}{\log(n)} \epsilon^{-\beta} b_n^{-\beta} = 
     C \epsilon^{-\beta} \log(n)^{-\beta}, 
 \end{eqnarray*}
 where the last inequality comes from \refeq{queue} and  Theorem 6.9 of \cite{LGR}. 
 Therefore, if $\left(\left(\frac{Z_{[nt]}}{b_n}  \right)_{t\in [0;1]} 
\right)_{n\ge 2}$ converges weakly to 
 $\left(\tilde Y_t\right)_{t\in [0,1]}$, 
the same is true for $\left(\left(\frac{Z_n(t)}{b_n}  \right)_{t\in [0;1]}
\right)_{n\ge 2} $. Using the fact that the sequence 
$\left(\left(\frac{Z_n(t)}{b_n}  \right)_{t\in [0;1]}\right)_{n\ge 2}$ is
 a sequence in the space $\CC([0,1])$ and that the Skorohod $J_1$-topology coincides 
 with the uniform one when restricted to $\CC([0,1])$, one deduces that $\left(\frac{Z_n(t)}{b_n}  \right)_{t\in [0;1]}$ converges weakly in $\CC([0,1])$, and that the limiting
 process $\left(\tilde Y_t\right)_{t\in [0,1]}$ is therefore continuous, which is false 
as soon as $\beta < 2$. 
 \end{proof}
\section{Proof of the local limit theorem in the lattice case}\label{sec:TLL}
\subsection{The event $\Omega_n$.}
\label{sec:omega_n} 
Set
\[ N_n^* := \sup_y N_n(y) \quad \textrm{and} \quad R_n := \#\{y\ :\ N_n(y)>0\} \, .
\]
\begin{lem}\label{lem:omega_n}
For every $n\ge 1$ and $1>\gamma > 0$, set
\[
\Omega_n=\Omega_n(\gamma) := \acc{R_n \le \frac{n }{(\log\log(n))^{1/4}}  
\ \mbox{and}\ N_n^*\le n^{\gamma}  }.
\]
Then, $\PP(\Omega_n) =1 - o(b_n^{-1})$.
Moreover, the following also holds on $\Omega_n$: 
\begin{eqnarray}
\label{minVn}
(\log\log(n))^{1/4} \le N_n^*\ \ \mbox{and}\ \ 
V_n\ge n^{1-\gamma(1-\beta)_+}.
\end{eqnarray}
\end{lem}

\begin{proof}

We first prove that 
\begin{equation}
\label{BSRange0}
\PP\pare{ R_n \geq n (\log\log (n))^{-1/4} } 
        =o(b_n^{-1}).
\end{equation}
Let us recall that for every $a, b \in \NN$, we have 
\begin{equation} 
\label{SousaddRange0}
\PP(R_n \geq a+b) \leq \PP(R_n \geq a) \PP(R_n \geq b) \, .
\end{equation} 
The proof is given for instance in \cite{Chen}.
We will moreover use the fact that ${\mathbb E}[R_n]\sim cn(\log(n))^{-1}$
and $Var(R_n)=O\left(n^2\log^{-4}(n)\right)$ (see \cite{LGR}).
Hence, for $n$ large enough, there exists $C>0$ such that we have 

$\displaystyle
\PP\left( R_n \geq \frac n{(\log\log (n))^{1/4}}   \right) 
 \leq  \PP\left( R_n \geq \floor{ \frac{n (\log\log(n))^{1/4}}{\log(n)}} \right)
       ^{\floor{ \log(n)(\log\log(n))^{-1/2} }  }\\
$
\begin{eqnarray*} 
& \leq & \PP\left( |R_n-{\mathbb E}[R_n]| \geq \frac 12
     \floor{ \frac{n (\log\log(n))^{1/4}}{\log(n)} }
            \right)
       ^{\floor{ \log(n)(\log\log(n))^{-1/2} }  }\\
& \leq & \left(\frac{5 Var(R_n)\log^2(n)}{n^2 (\log\log(n))^{1/2}} \right)
       ^{\floor{ \log(n)(\log\log(n))^{-1/2} }  }\\
& \leq & \left(\frac{C n^2\log^2(n)/\log^4(n)}{n^2 \sqrt{\log\log(n)}} \right)
       ^{\floor{ \log(n)(\log\log(n))^{-1/2} }  }\\
& \leq & \left(\frac{C} {(\log(n))^2} \right)
       ^{\floor{ \log(n)(\log\log(n))^{-1/2} }  }=
  \exp\left(-\log(n) \sqrt{\log\log(n)}\left(1-\frac{\log(C) }{2\log\log(n)}\right)\right).
\end{eqnarray*} 
This ends the proof of \refeq{BSRange0}. 

Let us now prove that
\begin{equation}
\label{N*_n}
\PP\cro{ N^*_n \geq  n^\gamma  } =o(b_n^{-1}).
\end{equation}
We have
\begin{eqnarray*} 
\PP(N^*_n \geq n^{\gamma}) 
& \leq & 
\sum_x \PP(N_n(x) \geq  n^{\gamma}) 
\\
& =  & \sum_x \PP(T_x \le n ; N_n(x) \geq  n^{\gamma}) \, , \,\,
\mbox{where } T_x := \inf\acc{n > 1, \mbox{ s.t. } S_n =x}\, ,
\\
& \le & \sum_x \PP(T_x \le n) \PP(N_n(0) \ge n^{\gamma}) 
\\
& \le &
\EE[R_n] \PP(T_0 \le n)^{n^{\gamma}}.
\end{eqnarray*} 

Hence, \refeq{N*_n} follows now from  $\EE[R_n] \sim cn(log(n))^{-1}$, and 
 from   $\PP(T_0 > n) \sim C/\log(n)$.
 
Since $n= \sum_y N_n(y) \le R_nN_n^*$, 
we get that $N_n^*\ge \frac n{R_n}\ge
\pare{\log\log(n)}^{1/4}$ on $\Omega_n$. 

To prove the lower bound for $V_n$, note that for $\beta \ge 1$, 
$V_n = \sum_y N_n(y)^{\beta} \ge \sum_y N_n(y)=n$. For 
$\beta < 1$,  on $\Omega_n$, 
\[
n= \sum_y N_n(y)= \sum_y N_n(y)^{\beta} N_n(y)^{1-\beta}  \le   
V_n  (N_n^*)^{1-\beta} \le V_n n^{\gamma(1-\beta)} \, .
\]
\end{proof}

%
%
\subsection{Scheme of the proof.}\label{sec:scheme}
It is easy to see (cf the proof of lemma 5 in \cite{BFFN})  that 
$\PP\pare{Z_n = \floor{b_nx}} = 0$ if $\PP\pare{n \xi_0 - \floor{b_nx} 
\notin d_0 \mathbb{Z}}=1$, and that if $\PP\pare{n \xi_0 - \floor{b_nx}
\in d_0 \mathbb{Z}}=1$, 
\[ \PP\pare{Z_n = \floor{b_nx}} = \frac {d_0}{2\pi}
\int_{-\frac{\pi}{d_0}}^{\frac\pi {d_0}}
e^{-it \floor{b_n x}}
{\mathbb E}\left[\prod_y \varphi_\xi(tN_n(y))\right]
\, dt\, .
\]
In view of lemma \ref{lem:omega_n}, 
we have to estimate 
$$
\frac {d_0}{2\pi}\int_{-\frac{\pi}{d_0}}^{\frac\pi {d_0}}
e^{-it \floor{b_n x}}
{\mathbb E}\left[\prod_y \varphi_\xi(tN_n(y)){\bf 1}_{\Omega_n}\right]
\, dt\, .
$$
This is done in several steps 
presented in the following propositions. 

\begin{prop}\label{lem:equivalent} 
Let $\gamma\in (0,1/(\beta+1))$ and $\delta\in (0,1/(2\beta))$ s.t. 
$\gamma \frac{(1-\beta)_+}{\beta}<\delta< 1/\beta-\gamma$. 
Then, we have
\[
\frac {d_0} {2\pi} \int_{\{|t| \le  n^\delta/b_n \} }
e^{-it \floor{b_n x}}  \EE \cro{\prod_y \varphi_{\xi}(tN_n(y))
   {\bf 1}_{\Omega_n} } \, dt 
= d_0\frac{C(x)}{b_n} + o( b_n^{-1}) \, ,
\]
uniformly in $x\in\mathbb R$.
\end{prop}

Recall next that the characteristic function $\phi$ of the limit distribution 
of $\left(n^{-1/\beta}\sum_{k=1}^n\xi_{ke_1}\right)_n$ 
has the following form~: 
$$\phi(u)=e^{- |u|^\beta (A_1+iA_2sgn(u)) },$$
with $0<A_1<\infty$ and $|A_1^{-1}A_2|\le|\tan(\pi\beta/2)|$. 
It follows that the characteristic function $\varphi_\xi$ of $\xi_0$ satisfies: 
\begin{eqnarray}
\label{phi0}
1-\varphi_\xi(u)\sim  |u|^\beta (A_1+iA_2sgn(u)) \quad \textrm{when } u\to 0.
\end{eqnarray}
Therefore there exist constants $\varepsilon_0>0$ and $\sigma>0$ such that
\begin{eqnarray}
\label{majorationphi} 
\max(\vert\phi(u)\vert,\vert\varphi_\xi(u)\vert)\le
\exp\left(-\sigma|u|^\beta\right) \quad \textrm{for all }u\in [-\varepsilon_0,\varepsilon_0].
\end{eqnarray}

\noindent Since $\overline{\varphi_{\xi}(t)} =
 \varphi_{\xi}(-t)$ for every $t\ge 0$, the  following propositions 
achieve the proof of Theorem \ref{thmTLL}:

\begin{prop}\label{sec:step1}
Let $\delta$ and $\gamma$ be as in Proposition \ref{lem:equivalent}. 
Then there exists $c>0$ such that
$$
\int_{ n^\delta/b_n }^{\varepsilon_0 n^{-\gamma} }
 {\mathbb E}\left[\prod_y \vert \varphi_\xi(tN_n(y))\vert {\bf 1}_{\Omega_n}\right]
\, dt= o(e^{-n^c}).$$
\end{prop}


\begin{prop}\label{sec:step3}
There exists $c>0$ such that 
$$\int_{\varepsilon_0n^{-\gamma}}^{\frac \pi {d_0}}
{\mathbb E}\left[\prod_y \vert\varphi_\xi(tN_n(y))\vert{\bf
1}_{\Omega_n}\right] \, dt
=o(e^{-n^c}).$$
\end{prop}

\subsection{Proof of Proposition
\ref{lem:equivalent}.}
Remember that $V_n=\sum_{z\in{\mathbb Z}^d}N_n^\beta(z)$.
We start by a preliminary lemma.
\begin{lem}\label{lem:borne}
\begin{enumerate}
\item If $\beta > 1$, 
$\sup_n {\mathbb E}\left[\pare{ \frac{n \log(n)^{\beta -1} }{V_n}}^{1/(\beta-1)}
      \right]<+\infty$.
\item If $\beta \le  1$, $\forall p \in \NN$, 
$\sup_n {\mathbb E}\left[\pare{ \frac{n \log(n)^{\beta -1} }{V_n}}^{p}
      \right]<+\infty$.
\end{enumerate}
\end{lem}

\begin{proof}
For $\beta > 1$, using H\"older's inequality with $p = \beta$, we get 
$$n = \sum_x N_n(x) \le V_n^{\frac 1 \beta} R_n^{\frac {\beta-1} \beta}$$
which means that
$$
 \pare{\frac{n \log(n)^{\beta -1} }{V_n}}^{1/(\beta-1)} \le \frac{\log(n) R_n}{n}.
$$
But it is proved in \cite{LGR} Equation (7.a) that $\EE[R_n]=\O( n/ \log(n))$. The result follows.

The result is obvious for $\beta =1$. 
For $\beta < 1$, H\"older's inequality with $p = 2-\beta$ yields
\[ 
n = \sum_x N_n^{\frac{\beta}{2-\beta}}(x) N_n^{\frac{2(1-\beta)}{2-\beta}}(x)
\le V_n^{\frac 1 {2-\beta}} \pare{\sum_x N_n^2(x)}^{\frac{1-\beta}{2-\beta}}
\]
and so
\[ \frac{n \log(n)^{\beta-1}}{V_n} \le
\pare{\frac{\sum_x N_n^2(x)}{n \log(n)}}^{1-\beta} \, .
\]
It is therefore enough to prove that there exists $c > 0$ such that 
\begin{equation}
\label{mom-exp}
\sup_n \EE\cro{\exp\pare{c \frac{\sum_x N_n^2(x)}{n \log(n)}}} < \infty
.
\end{equation}
Note that $\sum_x N_n^2(x) = \sum_{k=0}^{n-1} N_n(S_k)$. By Jensen's
inequality, we get thus
\[
\EE\cro{\exp\pare{c \frac{\sum_x N_n^2(x)}{n \log(n)}}}
\le \frac{1}{n} \sum_{k=0}^{n-1} \EE \cro{\exp\pare{c \frac{N_n(S_k)}{\log(n)}}}
\, .
\]
Observe now that $N_n(S_k)= \sum_{j=0}^k {\bf 1}_{\{S_k-S_j=0\}} 
+  \sum_{j=k+1}^{n-1} {\bf 1}_{\{S_j-S_k=0\}} \stackrel{(d)}{=} 
N_{k+1}(0) + N'_{n-k}(0) - 1$, where $(N'_n(x), n \in \NN, x\in \ZZ^d)$ is 
an independent copy of $(N_n(x), n \in \NN, x\in \ZZ^d)$. Hence, 
\[\EE\cro{\exp\pare{c \frac{\sum_x N_n^2(x)}{n \log(n)}}} 
\le \EE \cro{\exp\pare{c\frac{N_n(0)}{\log(n)}}}^2 \, .
\]
But, $\forall t > 0$, 
\[ \PP \pare{ N_n(0) \ge t \log(n)} \le \PP\pare{T_0 \le n}^
{\ceil{t \log(n)}}
\, ,
\]
and 
\[ \EE \cro{\exp\pare{c\frac{N_n(0)}{\log(n)}}} 
\le 1 + \int_0^{\infty} c \exp(ct) \exp\pare{-\ceil{t \log(n)} \PP(T_0 > n)} \, dt
\, .
\] 
Now \refeq{mom-exp} follows then from the fact that $\exists C > 0$ such that 
$\PP(T_0 > n) \sim C/\log(n)$ for any integer $n \ge 1$. 
\end{proof}

\noindent The next step is 
\begin{lem}\label{lem:gaussian}
Under the hypotheses of Proposition \ref{lem:equivalent}, we have
\begin{eqnarray*}
 \int_{\{|t| \le  n^{\delta}/b_n\}} 
e^{-it\floor{b_n x}} \EE \cro{\left\{\prod_y \varphi_\xi(tN_n(y))
-e^{-|t|^\beta(A_1+iA_2 sgn(t))  V_n }
\right\}
   {\bf 1}_{\Omega_n}} \, dt
= o( b_n^{-1}) \, ,
\end{eqnarray*} 
uniformly in $x\in{\mathbb R}$. 
\end{lem}
\begin{proof}
It suffices to prove that 
\[
 \int_{\{|t| \le n^{\delta}/b_n\}}{\mathbb E}[|E_n(t)| {\bf 1}_{\Omega_n}]
 \, dt = o(b_n^{-1})
 \] 
with
$$E_n(t):=\prod_y\varphi_\xi(tN_n(y))-
\prod_y \exp\left(-|t|^\beta N_n^\beta (y) (A_1+iA_2 sgn(t)) \right). $$
Observe that 
\begin{eqnarray*}
E_n(t) = \sum_y & & \left(\prod_{z<y}\varphi_\xi(tN_n(z))\right)
\left(\varphi_\xi(tN_n(y))- e^{-|t|^\beta N_n^\beta (y)(A_1+iA_2 sgn(t))  }\right)\\ 
         & \times & \left(\prod_{z>y}e^{- |t|^\beta N_n^\beta (z) (A_1+iA_2 sgn(t))}\right)\, , 
\end{eqnarray*} 
where an arbitrary ordering of sites of $\ZZ^d$ has been chosen. 
But on $\Omega_n$, if $|t|\le n^\delta b_n^{-1}$, then
\begin{eqnarray}
\label{tnnz}
|t| N_n(z)  \le  n^{\gamma+\delta} b_n^{-1}.
\end{eqnarray} 
Since $\gamma+\delta<\beta^{-1}$, 
this implies in particular that $|t| N_n(z)< \varepsilon_0$ for $n$ large enough. 
Thus, by using \eqref{majorationphi}, we get
\begin{eqnarray*} 
|E_n(t)| \le  \sum_y\left\vert \varphi_\xi(tN_n(y))-
   \exp\left(- |t|^\beta N_n^\beta(y) (A_1+iA_2 sgn(t))  \right)\right\vert
\exp\left(-\sigma |t|^\beta\sum_{z\ne y}N_n^\beta(z) \right),
\end{eqnarray*}
for $n$ large enough.  
Observe next that \eqref{phi0} implies 
$$\left\vert \varphi_\xi(u)-\exp\left(- |u|^\beta (A_1+iA_2 sgn(u))  \right)\right\vert\le
|u|^\beta h(\vert u\vert) \quad \textrm{for all } u\in \RR,$$ 
with $h$ a continuous and monotone function on $[0,+\infty)$ vanishing in $0$. 
Therefore by using \eqref{tnnz} we get
\begin{eqnarray*}
|E_n(t)|\le   |t|^\beta h( n^{\gamma+\delta} b_n^{-1}) 
 \sum_y N_n^\beta(y) \exp\left(-\sigma|t|^\beta \sum_{z\ne
y}N_n^\beta(z)\right).
\end{eqnarray*}
Now, according to \refeq{minVn} and since $\gamma<\frac{1}{\beta+1} 
\le \frac{1}{\beta + (1-\beta)_+}$,
if $n$ is large enough, we have on $\Omega_n$ 
$$ \sum_{z\ne y}N_n^\beta(z)
\ge V_n/2\qquad \text{for all }y\in \ZZ.$$
By using this and the change of variables
$v=tV_n^{1/\beta}$, we get 
$$
\int_{\{|t| \le n^\delta b_n^{-1} \}}\mathbb{E}\left[\vert E_n(t)\vert
{\bf 1}_{\Omega_n}\right]\, dt\leq h( n^{\gamma+\delta} b_n^{-1})
\mathbb{E}[V_n^{-1/\beta}]
\int_{\mathbb{R}} \vert v\vert^\beta \exp\left(-\sigma \vert v\vert^\beta /2\right)\, dv
=o(\mathbb{E}[V_n^{-1/\beta}]),
$$
which proves the result according to Lemma \ref{lem:borne}.
\end{proof}
 
\noindent Finally Proposition \ref{lem:equivalent} follows from the

\begin{lem}\label{le13} Under the hypotheses
of Proposition \ref{lem:equivalent}, we have
\[ \frac{d_0}{2 \pi} \int_{\{|t| \le n^{\delta}b_n^{-1}\}} e^{-it \floor{b_n x}}
\EE \cro{ e^{- |t|^\beta V_n (A_1+iA_2 sgn(t)) }{\bf 1}_{\Omega_n}} \, dt
=d_0 \frac{C(x)}{b_n}+ o( b_n^{-1}) \, ,
\]
uniformly in $x\in\mathbb R$. 
\end{lem}
\begin{proof}
Set 
$$I_{n,x}:=\int_{\{|t| \le n^\delta b_n^{-1}\}} e^{-it \floor{b_n x}}
 e^{- |t|^\beta V_n (A_1+iA_2 sgn(t))} \, dt,$$
which can be rewritten
$$I_{n,x}=\int_{\{|t| \le n^\delta b_n^{-1}\}} e^{-it \floor{b_n x}}
 \phi( t V_n^{1/\beta} ) \, dt.$$
Since $|\floor{b_n x}-b_n x|\le 1$, for all $n$ and $x$, it is immediate that 
$$I_{n,x}= \int_{\{|t| \le n^\delta b_n^{-1}\}} e^{-it b_n x}
    \phi( t V_n^{1/\beta} ) \, dt + \O(n^{2\delta}b_n^{-2}).$$
But $\delta<(2\beta)^{-1}$ by hypothesis. So actually
$$I_{n,x}= \int_{\{|t| \le n^\delta b_n^{-1}\}} e^{-it b_n x}
    \phi( t V_n^{1/\beta} ) \, dt + o( b_n^{-1}).$$
Next, with the change of variable $v=t b_n$, we get: 
\begin{eqnarray}
\label{Jn}
  \int_{\{|t| \le n^\delta b_n^{-1}\}} e^{-it b_n x}
    \phi( t V_n^{1/\beta} ) \, dt =b_n^{-1}\left\{ V_n^{-1/\beta}
    b_n f(x V_n^{-1/\beta}b_n) - J_{n,x}\right\},
\end{eqnarray}
where $f$ is the density function of the distribution with characteristic function $\phi$
and where
$$ J_{n,x}:=  \int_{\{|v|\ge n^{\delta}\} }\!\! 
            e^{-iv x} \phi( v b_n^{-1} V_n^{1/\beta}) \, dv.
$$
By lemma \ref{tech1} (applied with $m = 1$, $t_1=\theta_1=1$, 
$\gamma = \beta$), 
 $(W_n:=b_n V_n^{-1/\beta})_n$ 
converges almost surely, as $n\to \infty$, to 
the constant $\Gamma(\beta+1)^{-1/\beta} (\pi A)^{1-1/\beta}$. Moreover, Lemma \ref{lem:borne} 
ensures that the sequence $(W_n,n\ge 1)$ is uniformly integrable, 
so actually the convergence holds in ${\mathbb L}^1$. 
Let us deduce that 
\begin{eqnarray}
\label{WnW} 
{\mathbb E}[g_x(W_n)]={\mathbb E}[g_x(W)]
+o(1),
\end{eqnarray} 
where $g_x:z\mapsto z f(xz)$ and the $o(1)$ is uniform in $x$. First 
\begin{eqnarray*}
\left\vert{\mathbb E}[g_x(W_n)]- {\mathbb E}[g_x(W)] \right\vert
  &\le &   \sup_{x,z\in\mathbb R}\vert (g_x)'(z) \vert{\mathbb E}[\vert W_n- W \vert]\\
  &\le & \sup_u |f(u)+uf'(u)|  {\mathbb E}[\vert W_n- W\vert].
\end{eqnarray*}
This proves \eqref{WnW}. We observe that ${\mathbb E}[g_x(W)]=C(x)$.

\noindent In view of \eqref{Jn}, it only remains to prove that $\EE[J_{n,x}{\bf 1}_{\Omega_n}]=o(1)$ uniformly in $x$. 
But this follows from the basic inequality
$$
\EE[|J_{n,x}{\bf 1}_{\Omega_n}|]\le \int_{|v|\ge n^{\delta}}
  \EE\left[e^{-A_1 |v|^\beta \frac{V_n}{b_n^\beta} }{\bf 1}_{\Omega_n}\right] \, dv,
$$
and from the lower bound for $V_n$ given in \eqref{minVn} and from the 
choice $\delta > \gamma (1-\beta)_+/\beta$. 
\end{proof}

\subsection{Proof of Proposition
\ref{sec:step1}.}

Recall that on $\Omega_n$, $N_n(y) \le n^\gamma$, for all $y \in \ZZ^d$. 
Hence by \eqref{majorationphi},
\[
K_n:=\int_{ n^\delta/b_n }^{\varepsilon_0 n^{-\gamma} }
 {\mathbb E}\left[\prod_y \vert \varphi_\xi(tN_n(y))\vert {\bf 1}_{\Omega_n}\right]
\, dt \leq \int_{ n^\delta/b_n }^{\varepsilon_0 n^{-\gamma} }
{\mathbb E} 
\cro{ \exp \pare{-\sigma t^\beta V_n} {\bf 1}_{\Omega_n}} \, dt \, .
\]
With the change of variable $s=tV_n^{1/\beta}$, we get
\begin{eqnarray*}
K_n &\le& {\mathbb E}\left[V_n^{-1/\beta}\int_{ n^\delta V_n^{1/\beta}b_n^{-1} }
^{\varepsilon_0 n^{-\gamma}V_n^{1/\beta} }
\exp \pare{-\sigma s^\beta }\, ds {\bf 1}_{\Omega_n} \right]
\\
&\le&  \frac{1}{n^{\frac 1 \beta - \gamma \frac{(1-\beta)_+}{\beta}}}
\int_{ n^{\delta - \gamma \frac{(1-\beta)_+}{\beta}} \log(n)^{\frac{1-\beta}{\beta} }}^{+\infty}
\exp \pare{-\sigma s^\beta }\, ds \, ,
\end{eqnarray*}
which proves the proposition since $\delta>\gamma(1-\beta)_+/\beta$.
\subsection{Proof of Proposition
\ref{sec:step3}.}
We adapt the proof of \cite[Proposition 10]{BFFN}.
We will see that the argument of "peaks" still works here.
We endow ${\mathbb Z}^d$ with the ordered structure given
by the relation $<$ defined by
$$(\alpha_1,...,\alpha_d)<(\beta_1,...,\beta_d)\ 
  \leftrightarrow\ \exists i\in\{1,...,d\},\ \ \alpha_i<\beta_i,\
  \forall j<i,\ \alpha_j=\beta_j.$$
\medskip

We consider $\C^+=(x_1,...,x_T)\in({\mathbb Z}^d\setminus\{0\})^T$ for some positive
integer $T$ such that:
\begin{itemize}
\item $x_1+...+x_T=0$;
\item for every $i=1,...,T$, ${\mathbb P}(X_1=x_i)>0$; 
\item there exists $I_1\in\{1,...,T\}$ such that
\begin{itemize}
\item for every $i=1,...,I_1$, $x_i>0$,
\item for every $i=I_1+1,...,T$, $x_i<0$.
\end{itemize}
\end{itemize}
Let us write ${\mathcal C}^-:=(x_{T-i+1})_{i=1,...,T}$. 
We define $B:=\sum_{i=1}^{I_1}x_i$.
We observe that
$$ p:={\mathbb P}((X_1,...,X_T)=\C^+)
  ={\mathbb P}((X_1,...,X_T)=\C^-)>0.$$
We notice that $(X_1,...,X_T)=\C^+$ corresponds to 
a trajectory visiting $B$ only once before going back to the origin
at time $T$ (and without visiting $-B$).
Analogously, $(X_1,...,X_T)=\C^-$ corresponds
to a trajectory that goes down to $-B$ and 
comes back up to $0$ (and without visiting $B$),
and staying at a distance smaller than $\tilde d/2$ of the origin
with $\tilde d:=\sum_{i=1}^T|x_i|$ (where $|\cdot|$ is the absolute value
if $d=1$ and $|(a,b)|=\max(|a|,|b|)$ if $d=2$).
We introduce now the event 
$$\D_n:=\left\{C_n  > \frac{np}{2T}\right\},$$
where
$$C_n:=\#\left\{k=0,...,\left\lfloor \frac n T\right\rfloor-1\ :\ 
(X_{kT+1},\dots,X_{(k+1)T})=\C^\pm\right\}.$$
Since the sequences $(X_{kT+1},\dots,X_{(k+1)T})$, for $k\ge 0$, are independent of each other, 
Chernoff's inequality implies that there exists $c>0$ such that
$${\mathbb P}(\D_n) =1-o(e^{-cn}).$$
\noindent We introduce now the notion of "loop".
We say that there is a loop based on $y$ at time $n$ if $S_n=y$ and 
$(X_{n+1},\dots,X_{n+T})= \C^\pm$.
We will see (in Lemma \ref{sec:p_n} below) that, on $\Omega_n \cap \D_n$, there is a large number of $y\in \ZZ^d$
on which are based a large number of loops.
For any $y\in {\mathbb Z}^d$, let
$$
C_n(y):=\#\left\{k=0,\dots,\left\lfloor \frac n T\right\rfloor-1 \ :\ 
S_{kT}=y \textrm{ and }(X_{kT+1},\dots,X_{(k+1)T}) = \C^\pm\right\},
$$
be the number of loops based on $y$ before time $n$ (and at times which are multiple 
of $T$), and let 
$$p_n :=\#\left\{y\in {\mathbb Z}\ :\
C_n(y)
\ge \frac{\log\log(n)^{1/4}p}{4T}\right \},$$
be the number of sites $y\in \ZZ$ on which at least $a_n:=\left\lfloor
\frac{\log\log(n)^{1/4}p}{4T}\right\rfloor$
loops are based.

\begin{lem}\label{sec:p_n}
On $\Omega_n \cap \D_n$, we have, 
$p_n\ge c'n^{1-\gamma}$ with $c'=p/(4T)$. 
\end{lem}

\begin{proof}
Note that $C_n(y)\le N_n^*$ for all $y\in \ZZ^d$. Thus on $\Omega_n \cap \D_n$, 
we have
\begin{eqnarray*}
\frac{np}{2T} & \le & \sum_{ y \in {\mathbb Z}^d\ :\ C_n(y) < a_n}
C_n(y) + \sum_{ y \in {\mathbb Z}^d\ :\ C_n(y) \ge a_n}
C_n(y) 
\\
& \le & R_n a_n+ N_n^* p_n  \le   \frac {np}{4T} +p_n n^\gamma,
\end{eqnarray*}
according to lemma \ref{lem:omega_n}.
This proves the lemma. 
\end{proof}
We have proved that, if $n$ is large enough, the event $\Omega_n\cap\D_n$
is contained in the event 
$$\E_n:= \{p_n\ge c' n^{1-\gamma} \}.$$
Now, on  $\E_n$, we consider
$(Y_i)_{i=1,\dots,\left\lfloor   c'' n^{1-\gamma} \right\rfloor}$
(with $c'':=c'/(2\tilde d)$ if $d=1$ and with $c'':=c'/2\tilde d^2)$ if $d=2$)
such that
\begin{itemize}
\item on each  $Y_i$, at least $a_n$ loops are based,
\item for every $i,j$ such that $i\ne j$, we have
$\vert Y_i-Y_j\vert> \tilde d/2$.
\end{itemize}
For every $i=1,\dots,\left\lfloor  c'' n^{1-\gamma}  \right\rfloor$, 
let
$t_i^{(1)},\dots,t_i^{ (a_n)}$ be the 
$a_n$
first times (which are multiples of $T$) when a loop is based on the site $Y_i$. 
We also define $N_n^{0}(Y_i+B)$ as the
 number of visits of $S$ before time $n$ to $Y_i+B$, 
which do not occur during the time intervals 
$[t_i^{(j)},t_i^{(j)}+T]$, 
for $j\le a_n$.   

Since our construction is basically the same as in \cite[section 2.8]{BFFN},
the proof of the following lemma is exactly the same as the proof of \cite[Lemma 16]{BFFN}
and we do not prove it again.
\begin{lem} 
\label{independance} 
Conditionally to the event $\E_n$, 
$(N_n(Y_i+B)-N_n^{0}(Y_i+B))_{i\ge 1}$ is a sequence of independent
identically distributed random variables with binomial distribution
${\mathcal B}\left(a_n ;
\frac 12\right)$.
Moreover this sequence is independent of $(N_n^0(Y_i+B))_{i\ge 1}$. 
\end{lem}

\noindent 
Let $\eta$ be a real number such that $\gamma<\eta<(1-\gamma)/\beta$ (this is possible since 
$\gamma<1/(\beta +1)$). We define
$$\forall n\ge 1,\ \ d_n:= n^{-\eta} .$$
Let now 
$\rho:=\sup\{\vert\varphi_\xi(u)\vert\ :\ d\left(u,\frac{2\pi}{d_0}{\mathbb Z}\right)
\ge\varepsilon_0\}$. According to Formula 
\eqref{majorationphi} and since $\lim_{n\rightarrow\infty}d_n=0$, for $n$ large enough,
we have
\begin{eqnarray*} 
\va{\varphi_\xi(u)} 
& \leq & 
\rho {\bf 1}_{\{d\left(u,\frac{2\pi}{d_0}{\mathbb Z}\right) \geq \epsilon_0\}}
+ \exp\left(-\sigma d\left(u,\frac{2\pi}{d_0}{\mathbb Z}\right)^\beta \right)
{\bf 1}_{\{d\left(u,\frac{2\pi}{d_0}{\mathbb Z}\right) < \epsilon_0\}} 
\\
& \leq & \exp \pare{-\sigma d_n^\beta}
\, , 
\end{eqnarray*} 
as soon as $d\left(u,\frac{2\pi}{d_0}{\mathbb Z}\right) \geq d_n$.
Therefore, for $n$ large enough,
\begin{eqnarray}
\label{majodist}
\prod_z \va{\varphi_{\xi}(t N_n(z)) }
\leq \exp \pare{- \sigma d_n^\beta \# \acc{z\ :\ 
d\left(t N_n(z),\frac{2\pi}{d_0}{\mathbb Z}\right) \geq d_n }} \, .
\end{eqnarray}
\noindent Then notice that
\begin{eqnarray}
\label{Gt}
d\left(t N_n(z),\frac{2\pi{\mathbb Z}}{d_0}\right)\ge
d_n\ \Longleftrightarrow\ N_n(z)\in \I:=\bigcup_{k\in{\mathbb Z}}I_k,
\end{eqnarray}
where for all $k\in \ZZ$,
$$I_k:=\left[\frac{2k\pi}{d_0t}+\frac{d_n}
{t},\frac{2(k+1)\pi}{d_0 t}-\frac{d_n}
{t}\right]. $$
In particular ${\mathbb R}\setminus\I=\bigcup_{k\in{\mathbb Z}}J_k$, where for all $k\in \ZZ$,
$$J_k:=\left(\frac{2k\pi}{d_0t}-\frac{d_n}
{t},\frac{2 k\pi}{d_0 t}+\frac{d_n}{t}\right). $$

\begin{lem}\label{sec:liminf}
Under the hypotheses of Proposition \ref{sec:step3},
for every $i\le \left\lfloor c''n ^{1-\gamma} \right\rfloor$, 
$t\in (\varepsilon_0 n^{-\gamma},\pi/d_0)$ and $n$ large enough,
$${\mathbb P}\left(N_n(Y_i+B)\in \I\mid \E_n 
,\ N_n^0(Y_i+B)\right) \ge \frac 13\quad \textrm{almost surely}.$$
\end{lem}

\noindent 
Assume for a moment that this lemma holds true 
and let us finish  the proof of Proposition \ref{sec:step3}. Lemmas 
\ref{independance} and \ref{sec:liminf} ensure that conditionally to 
$\E_n$ and $((N_n^0(Y_i+B),i\ge 1)$, the events 
$\{N_n(Y_i+B)\in\I\}$, $i\ge 1$, are independent of each other, and all happen
with probability at least $1/3$.
Therefore, since $\Omega_n\cap\D_n\subseteq \E_n$,
there exists $c>0$, such that
$${\mathbb P}\left(\Omega_n \cap \D_n,
\ \#\{i\ :\ N_n(Y_i+B)\in{\mathcal I}\}\le
\frac{ c''n ^{1-\gamma}}{4} \right)
\le {\mathbb P}\left(B_n\le  \frac{ c''n ^{1-\gamma}}{4}\right) 
=o(\exp(-cn^{1-\gamma})),
$$
where for all $n\ge 1$, $B_n$ has binomial distribution
${\mathcal B}\left(\left\lfloor   c''n ^{1-\gamma}   \right\rfloor;
\frac 13\right)$.

\noindent But if $\#\{z\ :\ N_n(z)\in\I\}\ge
\frac{ c''n ^{1-\gamma}}{4}$,
then by \eqref{majodist} and \eqref{Gt} there exists a constant $c>0$, such that
$$\prod_z\vert\varphi_\xi(tN_n(z))\vert
\le \exp\left(-c n^{1-\gamma}d_n^\beta\right),
$$
which proves Proposition \ref{sec:step3} since $1-\gamma-\beta\eta>0$.

\begin{proof}[Proof of Lemma \ref{sec:liminf}]
First notice that by Lemma \ref{independance}, for any $H\ge 0$,
\begin{eqnarray}
\label{H}
{\mathbb P}(N_n(Y_i+B)\in\I \mid \E_n,\ N_n^0(Y_i+B)=H)={\mathbb P}\left(H+b_n\in\I\right),
\end{eqnarray}
where $b_n$ is a random variable with
binomial distribution
${\mathcal B}\left( a_n; \frac 12\right)$.
We will use the following result whose proof is postponed.

\begin{lem}\label{sec:lem0}
Under the hypotheses of Proposition \ref{sec:step3}, for
every $t\in\ ( \varepsilon_0 n^{-\gamma},\pi/d_0)$ and 
for $n$ large enough, the following holds:
\begin{itemize}
\item[(i)] For any integer $k$ such that all the elements of
$I_k-H$ are smaller than $\frac {a_n}2$, 
$${\mathbb P}(b_n\in (I_k-H))\ge
   {\mathbb P}(b_n\in (J_k-H)) . $$
\item[(ii)] For any integer $k$ such that all the elements of
$I_{k}-H$ are larger than $\frac {a_n}2 $, 
$${\mathbb P}(b_n\in (I_k-H))\ge
  {\mathbb P}(b_n\in (J_{k+1}-H)) . $$
\end{itemize}
\end{lem}

\noindent Now call $k_0$ the largest integer satisfying
the condition appearing in (i) and $k_1$ the smallest integer satisfying
the condition appearing in (ii).
We have $k_1=k_0+1$ or $k_1=k_0+2$. According to Lemma \ref{sec:lem0},
we have
\begin{eqnarray*}
{\mathbb P}\left(H+b_n\in\I\right)
& \ge & \sum_{k \le k_0} \PP\pare{H+b_n\in I_k} + 
		\sum_{k \ge k_1} \PP\pare{H+b_n \in I_k}
\\
&\ge &  \sum_{k \le k_0} \PP\pare{H+b_n \in J_k} 
	+  \sum_{k \ge k_1} \PP\pare{H+b_n \in J_{k+1}} 
\\
& = &   {\mathbb P}(H+b_n\not\in\I)-{\mathbb P}
(H+b_n\in J_{k_0+1}\cup J_{k_1}). 
\end{eqnarray*}
Hence, 
\[ \PP \pare{H+b_n \in \I} 
\geq \frac 1 2 \cro{ 1 - \PP(H+b_n\in J_{k_0+1}\cup J_{k_1})} \, .
\] 
Let $\bar{b}_n := {2} \pare{b_n - \frac{a_n}2}
{\sqrt{a_n}}$.
Since $\lim_{n\rightarrow +\infty}a_n=+\infty$, $(\bar{b}_n)_n$ converges
in distribution to a standard normal variable, whose distribution function
is denoted by $\Phi$. The interval $J_{k_1}$ being 
of length $2d_n/t$, 
\begin{eqnarray*} 
\PP(H+b_n \in J_{k_1}) & = & \PP(\bar{b}_n \in [m_n,M_n])
\, , \mbox{ with } M_n-m_n = 4\frac {d_n}{t\sqrt{a_n}} 
\\
& \leq & \Phi(M_n) - \Phi(m_n) + 
\frac{C}{\sqrt{a_n}}
\,  \mbox{ (by the Berry--Esseen inequality) }
\\
& \leq & \frac {M_n-m_n}{\sqrt{2\pi}} 
+ \frac{C}{\sqrt{a_n}} 
\\
& \leq & C' \frac{d_n}{\varepsilon_0 n^{-\gamma} \sqrt{a_n}} 
+ \frac{C}{\sqrt{a_n}} 
\, , 
\end{eqnarray*} 
for $t \ge \varepsilon_0 n^{-\gamma}$, and some constants $C>0$ and $C'>0$. 
Since $\lim_{n\rightarrow +\infty}a_n=+\infty$ and $\lim_{n\rightarrow +\infty}
d_n n^{\gamma}(a_n)^{-1/2}=0$ (since $\eta>\gamma$), 
we conclude that $\PP(H+b_n \in J_{k_1}) = o(1)$. The same holds for  
$\PP(H+b_n \in J_{k_0+1})$, so that for $n$ large enough, 
\[ \PP \pare{H+b_n \in \I} 
\geq \frac 1 2 \cro{ 1 - o(1)} \geq \frac 1 3 \, .
\]
Together with \eqref{H}, this concludes the proof of Lemma \ref{sec:liminf}. 
\end{proof}

\begin{proof}[Proof of Lemma \ref{sec:lem0}]
We only prove (i), since (ii) is similar.
So let $k$ be an integer such that all the elements of
$I_{k}-H$ are smaller than 
$\frac {a_n}2 $.
Assume that $(J_k-H)\cap \mathbb Z$ contains at least one nonnegative integer
(otherwise ${\mathbb P}(b_n\in (J_k-H))=0$  and there is nothing to prove).
Let $z_k$ denote the greatest integer in $J_k - H$, so that by our
assumption $\PP(b_n = z_k) > 0$ (remind that $0 \le z_k 
<\frac {a_n}2$). By monotonicity of the function
$z \mapsto  \PP(b_n=z)$, for $z\le \frac {a_n}2$, we get 
\[ 
\PP (b_n \in J_k - H) \leq \PP(b_n = z_k) \#((J_k -H)\cap \ZZ)
\leq \PP(b_n = z_k)  \ceil{ \frac{2d_n}t}
\, .
\]
In the same way,  
\[
\PP (b_n \in I_k - H) \geq  \PP(b_n = z_k) \#((I_k -H)\cap \ZZ)
\geq  \PP(b_n = z_k)  \floor{\frac {2\pi}{d_0 t} -  
\frac{2d_n}{t}}  
\, .
\]
Hence 
$$
\PP (b_n \in I_k - H) \geq \frac{ \floor {\frac{2\pi}{ d_0 t}-\frac{2 d_n}t} }   
{\ceil{\frac{ 2d_n}{t}}} 
\PP (b_n \in J_k - H)
\, .
$$
But $\pi/(d_0t)\ge 1$ and $\lim_{n\rightarrow +\infty}d_n=0$ by hypothesis. 
It follows immediately that 
for $n$ large enough, we have $2 d_n<\pi/(2 d_0)$, 
and so
$$
\floor{\frac{2\pi}{ d_0 t}-\frac{2 d_n}{t}}
\ge \floor{\frac {3\pi}{2d_0t}}\ge 1+\floor{\frac \pi{2d_0t}} \ge \ceil{\frac\pi{2d_0t}}\ge
\ceil{\frac{ 2d_n}{t}}\, .
$$
This concludes the proof of the lemma. 
\end{proof}

\section{Proof of the local limit theorem in the strongly nonlattice case}
As in \cite{BFFN}, the proof in the strongly nonlattice case is closely
related to the proof in the lattice case.

We assume here that $\xi$ is strongly nonlattice. In that case, 
there exist $\varepsilon_0 >0$,
$\sigma > 0$ and $\rho <1$ such that 
$|\varphi_{\xi}(u)| \leq \rho$ if $|u|\ge \varepsilon_0$ and
$|\varphi_{\xi}(u)| \leq \exp(-\sigma |u|^{\beta})$ if $|u|<\varepsilon_0$.

We use here the notations of Section \ref{sec:TLL}
with the hypotheses on $\gamma$, and $\delta$
of Proposition \ref{lem:equivalent}.
Let $h_0$ be the density of Polya's distribution:
$ h_0(y) = \frac1 \pi \frac{1-\cos(y)}{y^2}
$, with Fourier transform $\hat{h}_0(t) =(1 - |t|)_{+}$. For $\theta \in \mathbb{R}$,
let $h_{\theta}(y) = \exp(i \theta y) h_0(y)$ with Fourier transform
$\hat{h}_{\theta}(t)=\hat{h}_0(t+\theta)$. As in \cite[thm 5.4]{durrett},
it is enough to show
that for all $\theta \in \mathbb{R}$, 
\begin{equation}
\label{nonlattice.eq}
\lim_{n \rightarrow \infty} b_n \EE \cro{h_{\theta}(Z_n - b_n x)}
= C(x)\, \hat{h}_{\theta}(0) \, .
\end{equation} 
By Fourier inverse transform, we have
\[ 
b_n \EE \cro{h_{\theta}(Z_n - b_n x)}
= \frac{ b_n}{2 \pi} 
 \int_{\mathbb R} e^{-iub_n x}
\EE \cro{\prod_{x \in \ZZ^d} \varphi_{\xi}(u N_n(x))} \hat{h}_{\theta}(u)\; du
\, .
\]
Since $\hat{h}_{\theta}\in L^1$, we can restrict our study to the 
event $\Omega_n$ of Lemma \ref{lem:omega_n}. 
The part of the integral corresponding to $|u| \leq n^{ \delta}b_n^{-1}$
is treated exactly as in Proposition \ref{lem:equivalent}. The only change is that
we have to check that
\[
\lim_{n \rightarrow \infty} b_n  \int_{\{|u| \le n^{\delta}b_n^{-1}\} }
\EE \cro{ e^{-|u|^{\beta} V_n (A_1 + i A_2 sgn(u))}{\bf 1}_{\Omega_n}} 
(\hat{h}_{\theta}(u) - \hat{h}_{\theta}(0)) \, du    = 0
\, ,
\]
which is obviously true since  $V_n\ge n^{1-\gamma(1-\beta)_+}$
and since $2\gamma(1-\beta)_+<2\delta\beta<1$, 
using the fact that $\hat{h}_{\theta}$  is a Lipschitz function. 

Now, since $\hat{h}_{\theta}$ is bounded, the part corresponding to 
$n^{\delta}b_n^{-1} \leq |u| 
\leq \varepsilon_0 n^{-\gamma}$ is treated as in the proof of Proposition 
\ref{sec:step1} (since it only uses the behavior of 
$\varphi_{\xi}$ around $0$, which is
the same).

Finally, it remains to prove that 
\begin{equation}
\label{nonlattice.pf} 
\lim_{n \rightarrow  \infty} b_n
\int_{\{|u| \geq \varepsilon_0 n^{-\gamma} \}}
e^{-iu b_n x} \, \EE \cro{ \prod_x \varphi_{\xi}(u N_n(x)) 
{\bf 1}_{\Omega_n}}
\hat{h}_{\theta}(u) \, du = 0 \, .
\end{equation} 
We note that, if $|u| \geq \varepsilon_0 n^{-\gamma}$ and $x\in{\mathbb Z}^d$, we have
\begin{eqnarray*} 
\va{\varphi_{\xi}(u N_n(x))} 
& \leq & \exp(- \sigma |u|^{\beta} N_n^{\beta}(x)) \,\, 
{\bf 1}_{\{\va{u N_n(x)} \leq \varepsilon_0\}} 
+ \rho  \,\, {\bf 1}_{\{\va{u N_n(x)} \geq \varepsilon_0\}}  
\\
& \leq & \exp(- \sigma \varepsilon_0^\beta n^{-\gamma \beta} N_n^{\beta}(x))
\,\, {\bf 1}_{\{\va{u N_n(x)} \leq \varepsilon_0\}} 
+ \rho \,\, {\bf 1}_{\{\va{u N_n(x)} \geq \varepsilon_0\}}  
\, .
\end{eqnarray*} 
For $n$ large enough, $\rho \leq  
\exp(- \sigma \varepsilon_0^\beta n^{-\gamma \beta} )$. 
Therefore, if $n$ is large enough, then for 
all $x$ and $u$ such that $N_n(x) \geq 1$ and $|u| \geq \varepsilon_0 n^{-\gamma}$, 
we have
\[
\va{\varphi_{\xi}(u N_n(x))} 
\leq \exp(- \sigma \varepsilon_0^\beta n^{-\gamma \beta} ) \, .
\]
Hence,
\[
\va{\EE\cro{\prod_x \varphi_{\xi}(u N_n(x)) \, {\bf 1}_{\Omega_n}}}
\leq \EE\cro{\exp(- \sigma \varepsilon_0^\beta n^{-\gamma \beta} 
    R_n )  {\bf 1}_{\Omega_n}}
\leq \exp(- \sigma \varepsilon_0^\beta n^{1-\gamma(1+\beta)}) \, .
\] 
Therefore, since $\gamma(1+\beta) < 1$, we have
\[
\lim_{ n \rightarrow \infty} 
b_n \int_{\{|u| \geq  \varepsilon_0 n^{-\gamma} \}} 
e^{-iu b_n x} \, \EE\cro{\prod_x \varphi_{\xi}(uN_n(x)) \, 
{\bf 1}_{\Omega_n}} \hat{h}_{\theta}(u) \, du = 0 \, .
\]
This concludes the proof of Theorem \ref{thmTLL2}. 
\hfill $\square$

\noindent
{\bf Acknowledgments:}\\*
The authors are deeply grateful to Bruno Schapira for helpful and 
stimulating discussions.

\end{document}